\documentclass[10pt, article]{amsart}
\usepackage{tikz}
\usetikzlibrary{calc}
\usepackage{ae} % or {zefonts}
\usepackage[T1]{fontenc}
\usepackage[cp1250]{inputenc}
\usepackage{amsmath}
\usepackage{amssymb, amsfonts,amscd,verbatim}

\usepackage[normalem]{ulem}
\usepackage{hyperref}
\usepackage{indentfirst}
\usepackage{latexsym}
\input xy
\xyoption{all}

\usepackage{xcolor}

\usepackage{amsmath}    % need for subequations
\theoremstyle{plain}
\newtheorem{Pocz}{Poczatek}[section]
\newtheorem{Proposition}[Pocz]{Proposition}

\newtheorem{Theorem}[Pocz]{Theorem}
\newtheorem{Corollary}[Pocz]{Corollary}

\newtheorem{Lemma}[Pocz]{Lemma}
\newtheorem{Observation}[Pocz]{Observation}
\newtheorem{Notation}[Pocz]{Notation}
\newtheorem{Question}[Pocz]{Question}

\newtheorem{Example}[Pocz]{Example}

\theoremstyle{definition}
\newtheorem{Definition}[Pocz]{Definition}

\theoremstyle{remark}
\newtheorem{Remark}[Pocz]{Remark}

\newcommand{\N}{\mathbb{N}}

\numberwithin{equation}{section}
\title[Ends of spaces via linear algebra]
{Ends of spaces via linear algebra}

\author{Jerzy Dydak}
\address{University of Tennessee, Knoxville, TN 37996, USA}
\email{jdydak@utk.edu}
\address{Xi'an Technological University, No.2 Xuefu zhong lu, Weiyang district, Xi'an, China 710021}
\email{jdydak@gmail.com}

\author{Hussain Rashed}
\address{University of Tennessee, Knoxville, TN 37996, USA}
\email{hrashed1@vols.utk.edu}

\date{ \today
}


\begin{document}
\maketitle
\begin{center}
\today
\end{center}

\tableofcontents

\begin{abstract}
We develop a theory that may be considered as a prequel to the coarse theory.
We are viewing ends of spaces as extra points at infinity.
In order to discuss behaviour of spaces at infinity one needs a concept (a measure) of approaching infinity. The simplest way to do so is to list subsets of $X$ that are bounded (i.e. far from infinity) and that list should satisfy certain basic properties.
Such a list $S_X$ we call a \textbf{scale} on a set $X$ (see Section \ref{ScaledBooleanAlgebras}). In order to use ideas from the Stone Duality Theorem we consider sub-Boolean algebras $BA_X$ of the power set $2^X$ of $X$ that contain $S_X$ and that leads naturally to the concept of ends of a \textbf{scaled Boolean algebra} $(X,S_X,BA_X)$ which can be attached to $X$
and form a new scaled Boolean algebra $(\bar X,S_X,\overline{BA_X})$
that is \textbf{compact at infinity}.

Given a scaled space $(X,S_X)$ the most natural scaled Boolean algebra is $(X,S_X,2^X)$ which can be too far removed from the geometry of $X$. Therefore we need to figure out how to trim $2^X$ to a smaller sub-Boolean algebra $BA_X$. More generally, how to trim a sub-Boolean algebra $BA_X$ to a smaller one.
That is done using ideas from linear algebra. Namely, we consider a family
$\mathcal{F}$ of naturally arising
$S_X$-linear operators on $BA_X$ and the smaller sub-Boolean algebra
$BA_{\mathcal{F}}$ consists of eigensets of $\mathcal{F}$, an analog of eigenvectors from linear algebra. We show that all ends defined in literature so far (Freundenthal ends, ends of finitely generated groups, Specker ends, Cornulier ends, ends of coarse spaces) are special cases of such a process.

\end{abstract}

\section{Introduction}

Historically, as noted in \cite{DK} on p.287, 
ends are the oldest coarse topological notion. Here is their internal description:

\begin{Definition}\label{DefFreundenthalspace}
A \textbf{Freundenthal space} is a $\sigma$-compact, locally compact, connected and locally connected Hausdorff space $X$. 
\end{Definition}

\begin{Definition}\label{DefFreundenthalEnds}
A \textbf{Freudenthal end} of a Freundenthal space is a decreasing sequence $(U_i)_{i\ge 1}$ of components of $X\setminus K_i$, where $(K_i)_{i\ge 1}$ is an \textbf{exhausting sequence}, i.e. $K_i$ is compact, $K_i\subset int(K_{i+1})$ for each $i\ge 1$, and $\bigcup\limits_{i=1}^\infty K_i=X.$
The set of ends of $X$ is denoted by $Ends(X)$.
\end{Definition}

 Ends were used by Freudenthal in 1930 in his famous compactification (see \cite{Peschke}  for information about theorems in this section and see \cite{Engel} for results related to the theory of dimension):

\begin{Theorem}
Suppose $X$ is a $\sigma$-compact, locally compact, connected and locally connected Hausdorff space. It has a compactification $\bar X$ such that
$\bar X\setminus X$ is of dimension $0$ and $\bar X$ dominates any compactification
$\hat X$ of $X$ whose corona is of dimension $0$.
\end{Theorem}

\begin{Notation}
Given a $\sigma$-compact, locally compact, connected and locally connected Hausdorff space $(X,\mathcal{T})$. If $U\in \mathcal{T}$, we define the subset $U_{end}:=\{(U_i)\in Ends(X)|U_{i}\subset U$ for some $i\geq 1\}$ and $\widetilde{U}_{end}:=U_{end}\cup U$.
\end{Notation}
The family $\mathcal{T}\cup \{\widetilde{U}_{end}| U\in \mathcal{T}\}$ is a basis for a topology $\mathcal{T}_{end}$ on $X\cup Ends(X)$. The topological space $\bar X:=X\cup Ends(X)$ is a compactification of $X$ called the \textbf{Freudenthal compactification}. The space of ends $Ends(X)=\bar X\setminus X$ is of dimension 0, and $\bar X$ dominates any compactification
$\hat X$ of $X$ whose corona is of dimension $0$. Moreover, the number of ends of $X$ is the supremum of $n_i\geq 0$ where $n_i$ is the number of all mutually disjoint unbounded components of $X\setminus K_i$, for all $i\geq 1$.

Initially, ends were useful as properties of topological groups:
\begin{Theorem}
(Freudenthal) A path connected topological group has at most two ends.
\end{Theorem}

\begin{Theorem} (Leo Zippin \cite{Zippin})
If a locally compact, metrizable, connected topological group $G$ is two-ended, then $G$ contains a closed subgroup $T$ isomorphic to the group of reals such that the coset-space $G/T$ is compact; moreover, the space $G$ is the topological product of the axis of reals by a compact connected set homeomorphic to the space $G/T$.
\end{Theorem}

\begin{Theorem}
(H. Hopf) Let $G$ be a finitely generated discrete group acting on a space $X$ by covering transformations. Suppose the orbit space $B:=X/G$ is compact. Then (i) and (ii), below, hold.\\
(i) The space of ends of $X$ has 0, 1 or 2 (discrete) elements or is a Cantor space.\\ 
(ii) If $G$ also acts on $Y$ satisfying the hypotheses above, then $X$ and $Y$ have
homeomorphic end spaces.
\end{Theorem}

Conclusion (ii) suggests to regard the space of ends of $X$ as an invariant of the group $G$ itself:
\begin{Definition}
Let $p:X \to B$ be a covering map with compact base $B$ and the group of covering transformations $G$. The \textbf{end space} of $G$ is
$$Ends(G):= Ends(X).$$
\end{Definition}

When applied to a Cayley graph of $G$, it gives the standard definition of ends of finitely generated groups (see \cite{DK}, p.295). See \cite{Geog} for basic results in this theory and see \cite{Grom} for more general facts in coarse geometry related to groups. \cite{MM} contains interesting results for ends of finitely generated groups.

Stone \cite{Stone} assigns to each Boolean algebra $B$ a totally disconnected compact Hausdorff space $X_B$ called the Stone's space associated to $B$; and conversely, every Stone's space $X$ (totally disconnected compact Hausdorff space) is assigned a Boolean algebra $S_X$ that consists of all clopen subsets of $X$. Stone's representation theorems assure us that two Boolean algebras are isomorphic if and only if their corresponding Stone's spaces are homeomorphic; and that two Stone's spaces are homeomorphic if and only if their Boolean algebras of their clopen subsets are isomorphic. See \cite{HR} for basic results on Stone Duality.

Realizing that the space of ends of a finitely generated group is a Stone's space, and that its Boolean algebra of all clopen subsets is isomorphic to the quotient Boolean algebra of the Boolean algebra consists of all almost right invariant subsets by the ideal of finite subsets ($A\subset G$ is called almost  right invariant if the symmetric difference $A\Delta A\cdot g$ is finite for all $g\in G$); E. Specker \cite{Sp} defined the space of ends of an arbitrary group as the Stone's space of the quotient Boolean algebra of the Boolean algebra of all almost right invariant subsets by the ideal of finite subsets.

Adopting Specker's approach, W. Dicks and M. J. Dunwoody 
\cite{DD} consider ends of non-finitely generated groups. In particular, they proved the following result that is a generalization of the famous theorem of Stallings \cite{St}:

\begin{Theorem}
A group $G$ has infinitely many ends if and only if 
one of the following conditions holds:\\
(i) $G$ is countably infinite and locally finite,\\
(ii) $G$ can be expressed as an amalgamated free product $A\ast_CB$ or an HNN extension $A\ast_C$, where $C$ is a finite subgroup of $A$ and $B$ such that $[A : C]\ge 3$ and $[B : C]\ge 2$.
\end{Theorem}
Yves Cornulier \cite{Corn} has also studied the space of ends of infinitely generated groups.

When a group $G$ is equipped with a left-invariant proper metric, its almost right invariant subsets are exactly the coarsely clopen subsets of $G$ equipped with the large scale structure induced by its metric. The concept of coarsely clopen subsets is defined for any large scale space $X$. Furthermore, the family of all coarsely clopen subsets $\mathcal{CC}(X)$ of a given large scale space $X$ is a Boolean algebra. Consequently, it is natural to define the space of ends $Ends(X)$ of a coarsely connected large scale space $X$ as the Stone's space of the to the quotient Boolean algebra of $\mathcal{CC}(X)$ by the ideal of all bounded subsets. This is done in \cite{MRD}.
In this paper we generalize all above approaches to ends using ideas from Linear Algebra, namely analogs of eigenvectors of linear operators.

The authors declare that all data supporting the findings of this study are available within the article.

\section{Scaled spaces}

In order to discuss behaviour of spaces at infinity one needs a concept (a measure) of approaching infinity. The simplest way to do so is to list subsets of $X$ that are bounded and that list should satisfy certain basic properties. This leads us to the next concept.

\begin{Definition}
A \textbf{scaled space} is a pair $(X,S_X)$ of a set and a \textbf{scale} $S_X$ in $X$, i.e. a family $S_X$ of subsets of $X$ that is closed under finite unions and contains the empty set. We do not assume that $X=\bigcup S_X$, so $S_X$ may consist of the empty set only. A subset $A\subset X$ is called \textbf{bounded} if there exists $B\in S_X$ such that $A\subset B$; otherwise, $A\subset X$ is called \textbf{unbounded}. Given a scaled space $(X,S_X)$ and a subset $Y\subset X$, the family $B_Y:=\{K\cap Y:K\in S_X\}$ is a scale in $Y$. The pair $(Y,B_Y)$ is a \textbf{scaled subspace} of $(X,S_X)$.
\end{Definition}

\begin{Definition}
$C\equiv D$ mod $S_X$ means existence of $B_1,B_2\in S_X$ such that $C\cup B_1=D\cup B_2$. Equivalently, there is $B\in S_X$ such that
$C\setminus B=D\setminus B$ ($B:=B_1\cup B_2$ works).
\end{Definition}

If a scale $S_X$ is a \textbf{bornology}, i.e. every subset of $B\in S_X$ belongs to $S_X$, then
$C\equiv D$ mod $S_X$
is equivalent to the symmetric difference $C\Delta D$ belonging to $S_X$. For arbitrary scaled spaces, it is equivalent to the symmetric difference $C\Delta D$ being bounded. 

\begin{Observation}
The reason we do not assume each scale $S_X$ is a bornology is that when discussing topologies induced by certain families containing $S_X$ would quite often lead to discrete topologies and that is not what we are interested in.
\end{Observation}

\begin{Lemma}
If $C\equiv D$ mod $S_X$, then $X\setminus C\equiv X\setminus D$ mod $S_X$.
\end{Lemma}
\begin{proof}
Suppose $C\setminus B=D\setminus B$  for some $B\in S_X$.
Then $(X\setminus C)\cup B=(X\setminus D)\cup B$.
\end{proof}

\begin{Definition}
A \textbf{connectivity structure} $C_X$ on a set $X$ is a family of subsets of $X$ (called \textbf{connected}) containing all singletons with the property that any subset $Y$ of $X$
must belong to $C_X$ if every two points $x,y$ of $Y$ can be joined by a chain $C_i\in C_X$, $1\leq i\leq n$, i.e. $x\in C_1$, $y\in C_n$,
and $C_i\cap C_{i+1}\ne\emptyset$ for each $i < n$.
\end{Definition}

\begin{Proposition}
Suppose $(X,C_X)$ is a set equipped with a connectivity structure and $Y$ is a subset of $X$. Each point $y\in Y$ belongs to a maximal connected subset of $Y$ called the $C_X$-component of $Y$. The family of all $C_X$-components of $Y$ is a partition of $Y$.
\end{Proposition}
\begin{proof}
Let $y\in Y$ and $\mathcal{A}_y:=\{A\subset Y:y\in A  \text{ and } A\in C_X\}$. Clearly, $\mathcal{A}_y\neq\emptyset$ as $\{y\}\in\mathcal{A}_y$. Let $\mathcal{C}_y$ be a chain in $\mathcal{A}_y$, then $\bigcup\limits_{A\in \mathcal{C}_y}A$ is an upper bound of $\mathcal{C}_y$ and $\bigcup\limits_{A\in \mathcal{C}_y}A\in \mathcal{A}_y$. By Zorn's Lemma, $\mathcal{A}_y$ has a maximal.
\end{proof}

\begin{Definition}\label{SBAandCompactAtInfinity}
A \textbf{scaled Boolean algebra} is a triple $(X,S_X,BA_X)$ consisting of a set, a scale $S_X$, and a Boolean algebra $BA_X$ of subsets of $X$ such that $S_X\subset BA_X$.

Given a scaled Boolean algebra $(X,S_X,BA_X)$ and $Y\subset X$. The scaled Boolean algebra $(Y,B_Y,BA_Y)$ is a \textbf{scaled sub-Boolean algebra} of $(X,S_X,BA_X)$ if:\\
$1.$ $(Y,B_Y)$ is a scaled subspace of $(X,S_X)$.\\
$2.$ $\{A\cap Y\mid A\in BA_X\}=BA_Y$.
\end{Definition}

\begin{Proposition}\label{BAInducedByConnectivity}
Given a scaled space $(X,S_X,C_X)$ equipped with a connectivity structure.
If for each bounded set $B$ the union of bounded $C_X$-components
of $X\setminus B$ is bounded, then
the family $BA_X$ of all subsets of $X$ each equivalent to a union of components of $X\setminus B$ for some bounded $B$ is a $S_X$-scaled Boolean algebra.
\end{Proposition}
\begin{proof}
If $A\equiv C$ and $C$ is a union of some $C_X$-components of $X\setminus B$, then $X\setminus C$ is equivalent to the union of the remaining $C_X$-components of $X\setminus B$.

Suppose $A_i\equiv C_i$, $C_i$ being a union of $C_X$-components
of $X\setminus B_i$, $i=1,2$. Put $B:=B_1\cup B_2$
and consider all $C_X$-components of $X\setminus B$ contained in either $C_1$ or $C_2$. Since the family of all such $C_X$-components partitions $C_1\cup C_2$, their union $C$ is equivalent to $A_1\cup A_2$. Finally, consider all $C_X$-components of $X\setminus B$ contained in both $C_1$ and $C_2$, then their union $D$ is equivalent to $A_1\cap A_2$. This shows that if $A_1,A_2\in BA_X$, then $A_1\cup A_2,A_1\cap A_2\in BA_X$.
\end{proof}
The scaled Boolean algebra as in \ref{BAInducedByConnectivity}
is called the Boolean algebra induced by the connectivity structure $C_X$.

\section{Compactifications of scaled Boolean algebras}\label{ScaledBooleanAlgebras}
In this section we define an analog of a compactification of a topological space.

\begin{Definition}\label{DefCompactAtInfBA}
Given a scaled Boolean algebra $(X,S_X,BA_X)$, $x\in X$ is a \textbf{point at infinity} of $X$ if $\{x\}$ is unbounded.
A scaled Boolean algebra is \textbf{compact at infinity} if every family $\{A_i\}_{i\in I}$ of elements of the Boolean algebra $BA_X$ covering all points at infinity has
a finite subfamily $\{A_i\}_{i\in F}$ covering $X\setminus B$ for some $B\in S_X$.
\end{Definition}

\begin{Observation}
If a scaled Boolean algebra $(X,S_X,BA_X)$ is compact at infinity and has no points at infinity, then $X$ is bounded.
\end{Observation}
\begin{proof}
Since each point $x$ is bounded, there is $B(x)\in S_X$ containing $x$.
There are finitely many points $x_i$, $i\leq n$, so that $B:=X\setminus \bigcup\limits_{i=1}^n B(x_i)$ is bounded. Hence, $X$ is bounded.
\end{proof}

\begin{Observation}
Each scaled Boolean algebra $(X,S_X,BA_X)$ induces a topology on $X$ in which elements of the Boolean algebra $BA_X$ are clopen. If a scaled Boolean algebra is compact at infinity and every bounded union of elements of $BA_X$ belongs to $BA_X$, then each clopen set belongs to the $BA_X$.
\end{Observation}
\begin{proof}
Notice that $BA_X$ is a basis for a topology on $X$ in which every element of $BA_X$ is clopen.

Suppose $(X,S_X,BA_X)$  is compact at infinity and $A$ is clopen.
Express $A$ as a union of elements $A_i$, $i\in I$, of $BA_X$ and do the same for $X\setminus A$. The resulting cover has a finite subcover $\mathcal{S}$ of $X\setminus B$
for some $B\in S_X$. Notice that $K:=\bigcup\limits_{i\in I}(B\cap A_i)\in BA_X$, and hence $A$ is a finite union of elements of $BA_X$
(the union of those from $\mathcal{S}$ contained in $A$ and $K$). Therefore $A$ belongs to $BA_X$.
\end{proof}

\begin{Proposition}
Let $(X,S_X,BA_X)$ be a compact at infinity scaled Boolean algebra, and $(Y,B_Y,BA_Y)$ is a scaled sub-Boolean algebra of $(X,S_X,BA_X)$ with $Y\subset X$ being a closed subset (i.e. its complement being a union of elements of $BA_X$), then $(Y,B_Y,BA_Y)$ is compact at infinity.
\end{Proposition}
\begin{proof}
Let $\{A_i\}_{i\in I}$ be a family of elements of the Boolean algebra $BA_Y$ covering all points at infinity of $Y$. There exists a family $\{C_i\}_{i\in I}$ of elements of the Boolean algebra $BA_X$ such that $A_i=C_i\cap Y$ for all $i\in I$. Express $X\setminus Y$ as a union of elements $\{D_j\}_{j\in J}$ of $BA_X$, then the family $\{C_i, D_j:i\in I\text{ and } j\in J\}$ covers all points at infinity of $X$.
A finite subfamily of it covers all of $Y$ but a bounded subset of $Y$.
\end{proof}

\begin{Definition}
Let $(X,S_X,BA_X)$ be a scaled Boolean algebra. An \textbf{end} of it is a family $V=\{A_i\}_{i\in I}$ of unbounded elements of $BA_X$ that is maximal with respect to the property that $\bigcap\limits_{i\in F} A_i$ is unbounded for each finite subset $F$ of $I$.
 An end $V=\{A_i\}_{i\in I}$ is \textbf{external} if $\bigcap\limits_{i\in I} A_i$ is bounded, otherwise it is \textbf{internal}. The family of all ends of $(X,S_X,BA_X)$ is denoted by $Ends(X,S_X,BA_X)$ or by $Ends(X)$ for simplicity if the Boolean algebra is clearly understood. Similarly, the family of all internal (respectively, external) ends of $(X,S_X,BA_X)$ is denoted by $IntEnds(X,S_X,BA_X)$ (respectively, $ExtEnds(X,S_X,BA_X)$) or by $IntEnds(X)$ (respectively, $ExtEnds(X)$) for simplicity if the Boolean algebra is clearly understood.
\end{Definition}

\begin{Observation}
If $V=\{A_i\}_{i\in I}$ is an external end, then in fact
$\bigcap\limits_{i\in I} A_i=\emptyset$.
\end{Observation}
\begin{proof}
Pick $B\in S_X$ containing $\bigcap\limits_{i\in I} A_i$.
For each $i\in I$, $A_i\setminus B$ must belong to $V$
as it has unbounded intersections with any finite subfamily of $V$.
Therefore, $\bigcap\limits_{i\in I} A_i=\emptyset$.
\end{proof}

\begin{Observation}\label{ObsAboutInternalEnds}
If $V=\{A_i\}_{i\in I}$ is an internal end, then every point $x$ of
$C:=\bigcap\limits_{i\in I} A_i$ is a point at infinity of $X$.
Moreover, if $x\in C$ belongs to $A\in BA_X$, then $C\subset A$.
\end{Observation}
\begin{proof}
As above, $B\cap C=\emptyset$ for each $B\in S_X$, hence any $x\in C$ is unbounded. If $x\in A\in BA_X$, then for any finite $F\subset I$
the set $A\cap \bigcap\limits_{i\in F} A_i $ is unbounded
as it contains $x$. Thus, $A=A_i$ for some $i\in I$ and $C\subset A$.
\end{proof}

In view of Observation \ref{ObsAboutInternalEnds} the intersection of all elements of any internal end is a blob that should be treated as a singular point. That leads to the following:

\begin{Definition}
A scaled Boolean algebra $(X,S_X,BA_X)$ is \textbf{Hausdorff} if
each point at infinity $x$ of $X$ is the intersection of all elements of $BA_X$ containing it.
\end{Definition}

\begin{Proposition}
Suppose $(X,S_X,BA_X)$ is a scaled Boolean algebra.\\ 
a. If $(X,S_X,BA_X)$ is compact at infinity, then it has no external ends.\\
b. If $(X,S_X,BA_X)$ has no ends, then it is compact at infinity.
\end{Proposition}
\begin{proof}
a. Assume that $(X,S_X,BA_X)$ is compact at infinity. If  $V=\{A_i\}_{i\in I}$ is an external end of it, then $\{X\setminus A_i\}_{i\in I}$ is a cover of $X$ by elements of the Boolean algebra $BA_X$. Choose a finite subset $F\subset I$ such that $\{X\setminus A_i\}_{i\in F}$ covers $X\setminus B$ for some $B\in S_X$. Then $\bigcap\limits_{i\in F} A_i$ is bounded, a contradiction.\\
b. Suppose that $Ends(X)=\emptyset$ and that $(X,S_X,BA_X)$ is not compact at infinity. Choose a cover $\{A_i\in BA_X:i\in I\}$
of points at infinity of $X$ that has no finite subcover of $X\mod S_X$. Notice that, for any finite subset $F\subset I$, if $\bigcap\limits_{i\in F} X\setminus A_i$ is bounded, then $\bigcup\limits_{i\in F} A_i$ is a finite subcover of $X\mod S_X$, a contradiction. Therefore, $\{X\setminus A_i:i\in I\}$ is contained in some end, a contradiction.
\end{proof}

\begin{Definition}\label{DefCompactificationAtInfOfScBA}
A \textbf{compactification at infinity} of a scaled Hausdorff Boolean algebra $(X,S_X,BA_X)$ is a compact at infinity scaled Hausdorff Boolean algebra $(\tilde X,S_X,\widetilde{BA_X})$ such that the following conditions are satisfied:\\
1. $(X,S_X,BA_X)$ is a scaled sub-Boolean algebra of $(\tilde X,S_X,\widetilde{BA_X})$. \\
2. If $A\in \widetilde{BA_X}$ intersects $\tilde X\setminus X$,
then it intersects $X$ as well.
\end{Definition}

\begin{Observation}
Conditions 1 and 2 in \ref{DefCompactificationAtInfOfScBA} can be summarized as follows: the function $A\to A\cap X$
induces isomorphism of Boolean algebras $(\tilde X,S_X,\widetilde{BA_X})$ and $(X,S_X,BA_X)$.
\end{Observation}

\begin{Proposition}
Given a compactification at infinity $(\tilde X,S_X,\widetilde{BA_X})$ of a scaled Hausdorff Boolean algebra $(X,S_X,BA_X)$
and given $A\in BA_X$, the unique element $A'$ of $\widetilde{BA_X}$ such that $A=A'\cap X$ is equal $cl(A)$,
where $cl$ is the closure operator in $\tilde X$
induced by $\widetilde{BA_X}$. Moreover, for all $A_1,A_2\in BA_X$, then $cl(A_1\cap A_2)=cl(A_1)\cap cl(A_2)$.
\end{Proposition}
\begin{proof}
Since $A'$ is a closed subset of $\tilde X$ containing $A$, so $cl(A)\subset A'$.

If $x\notin cl(A)$, then there is $C\in \widetilde{BA_X}$ containing $x$ and not intersecting $A$. Look at $C\cap A'$. It does not intersect $A$, so $X\cap C\cap A'=\emptyset$. Hence $C\cap A'=\emptyset$ and $x\notin A'$. Finally, assume that $A_1,A_2\in BA_X$, then $A_i=X\cap cl(A_i)$, $i=1,2$. Since $cl(A_1)\cap cl(A_2)\in \widetilde{BA_X}$, and $A_1\cap A_2=X\cap(cl(A_1)\cap cl(A_2))$, thus $cl(A_1\cap A_2)=cl(A_1)\cap cl(A_2)$.
\end{proof}

\begin{Proposition}\label{ObsInducedExtEnds}
Given a compactification at infinity $(\tilde X,S_X,\widetilde{BA_X})$ of a scaled Hausdorff Boolean algebra $(X,S_X,BA_X)$
and given $x\in \tilde X\setminus X$, the family $V_x$ of all $A\in BA_X$ such that $x\in cl(A)$ is an external end of $(X,S_X,BA_X)$.
Moreover, if $x\ne y\in \tilde X\setminus X$, then $V_x\ne V_y$.
\end{Proposition}
\begin{proof}
If $x\in cl(A_i)$, $i\leq n$, then $C:=\bigcap\limits_{i=1}^n A_i$
being bounded implies $ \bigcap\limits_{i=1}^n cl(A_i)=C$,
a contradiction. If $A\in BA_X$ intersects each $A_i$ along an unbounded subset and $x\notin cl(A)$, then there is $A'\in BA_X$
containing $x$ in its closure
with $cl(A')\cap A=\emptyset$. Hence, $A\cap A'=\emptyset$, a contradiction.
\end{proof}

Our next result is the main point of this section and it is motivated by \ref{ObsInducedExtEnds}.
\begin{Lemma}\label{BasicPropOfAlgOfEnds}
Let $(X,S_X,BA_X)$ be a scaled Hausdorff Boolean algebra.\\
($a$) The family $\overline{BA_X}:=\{\bar A:A\in BA_X\}$ is a sub-Boolean algebra of $2^{\bar X}$, and $(\bar X,S_X,\overline{BA_X})$ is a scaled Hausdorff Boolean algebra, where $\bar A:=A\cup\{V\in ExtEnds(X):A\in V\}$ and $\bar X:=X\cup ExtEnds(X)$.\\
($b$) Let $A\subset X$ and $V\in ExtEnds(X)$. $A\in V$ if and only $V\in cl_{\bar X}(A)$.\\
($c$) If $A\in BA_X$, then $cl_{\bar X}(A)=\bar A$.\\
($d$) If $\{cl_{\bar X}(A_i):i\in I\}$ is a family of elements of $\overline{BA_X}$ covering $Ends(X)$, then there is a finite subset $F\subset I$ such that $X\setminus \bigcup\limits_{i\in F}A_i$ is bounded.
\end{Lemma}
\begin{proof}
($a$) Notice that $\overline {A_1\cup A_2}=\bar A_1\cup \bar A_2$ and $\overline {A_1\cap A_2}=\bar A_1\cap \bar A_2$ for all $A_1,A_2\in BA_X$. Hence, $\overline{X\setminus A}= \overline {X}\setminus \overline {A}$.
Moreover, $\bar B=B$, for all $B\in S_X$.\\
($b$)  Let $A\subset X$ and $V\in ExtEnds(X)$. Assume that $A\in V$ and let $C\in BA_X$ such that $\bar C$ is a neighborhood of $V$, then $A,C\in V$ and so $C\cap A\ne \emptyset$. Thus, $V\in cl_{\bar X}(A)$. Conversely, assume that $V\in cl_{\bar X}(A)$ and $A\notin V$. Choose $C\in V$ such that $C\cap A=\emptyset$. Hence, $\bar C$ is a neighborhood of $V$ that misses $A$, a contradiction.\\
($c$) Let $A\in BA_X$. Clearly $cl_{\bar X}(A)\subset\bar A$. If $V\in \bar A\cap ExtEnds(X)$, then $A\in V$, and hence by part ($b$), $V\in cl_{\bar X}(A)$.\\
($d$) Let $\mathcal{F}$ be the collection of all finite subsets of $I$. Seeking contradiction assume that for any $F\in \mathcal{F}$, $A_{F}=X\setminus \bigcup\limits_{i\in F}A_i$ is unbounded. The collection $\{A_{F}:F\in \mathcal{F}\}$ is contained in some end $V\in Ends(X)$. It cannot be internal, so it is external.
Hence, $\{A_{F}:F\in \mathcal{F}\}\subset V\in cl(A_j)$ for some $j\in I$ which implies that $A_j\in V$ and $X\setminus A_j\in V$, a contradiction. 
\end{proof}

\begin{Theorem}
Given a scaled Hausdorff Boolean algebra $(X,S_X,BA_X)$, the scaled Boolean algebra $(\bar X,S_X,\overline{BA_X})$ is the unique compactification at infinity of $(X,S_X,BA_X)$. Moreover, the space of ends $Ends(X)$ is compact and totally disconnected.
\end{Theorem}
\begin{proof}
Given a family $\{cl(A_i)\}_{i\in I}$ of all elements of $\overline{BA_X}$ that covers all points at infinity of $(\bar X,S_X,\overline{BA_X})$, then it must cover $Ends(X)$. By part ($d$) of \ref{BasicPropOfAlgOfEnds}, there is a finite $F\subset I$ such that $X\setminus(\bigcup\limits_{i\in F} A_i)$ is bounded. Since $Ends(X)\subset \bigcup\limits_{i\in F} cl(A_i)$, $\bar X\setminus(\bigcup\limits_{i\in F} cl(A_i))$ is bounded. This shows that $Ends(X)$ is compact and $(\bar X,S_X,\overline{BA_X})$ is  compact at infinity. \ref{BasicPropOfAlgOfEnds} tells us that $(\bar X,S_X,\overline{BA_X})$ is indeed a compactification at infinity and hence $Ends(X)$ is Hausdorff. The family $\{cl(A)\cap Ends(X):A\in BA_X\}$ is a basis for $Ends(X)$ consisting of clopen subsets.
\end{proof}

\subsection{Freundenthal compactification at infinity}
In this part we present a general theory of Freundenthal ends.
\begin{Definition}\label{DefFreundenthalScaledSpace}
A \textbf{Freundenthal scaled space} is a triple $(X,S_X,C_X)$,
where $S_X$ is a scale covering $X$ and $C_X$ is a connectivity structure, with the property that for any bounded set $B$ of $X$ there are only finitely many unbounded components of $X\setminus B$ and the union of remaining components of $X\setminus B$ is bounded.

If $BA_X$ is the Boolean algebra induced by the connectivity structure $C_X$, the  compactification at infinity $(\bar X,S_X,\overline{BA_X})$ is called the \textbf{Freundenthal compactification at infinity} of $(X,S_X,BA_X)$.
\end{Definition}

\begin{Proposition}\label{FreundenthalEndsCorollary}
If a triple $(X,S_X,C_X)$ is a Freundenthal scaled space and $BA_X$
is the induced Boolean algebra, then for each end $V$ of $BA_X$ and each
bounded set $B$ of $X$ there is a unique unbounded component $C$ of $X\setminus B$
belonging to $V$. Moreover, if $X=\bigcup\limits_{n=1}^\infty B_n$, where $(B_n)_{n\in\N}$ is an increasing sequence, 
and $B_n$ is a basis of bounded subsets of $X$, then ends of $BA_X$
are in one-to-one correspondence with decreasing sequences $C_n$,
where $C_n$ is an unbounded component of $X\setminus B_n$.
\end{Proposition}
\begin{proof}
Let $V$ be an end of $BA_X$, $B$ be a bounded subset of $X$, and $C_1,\dots, C_n$ be the unbounded components of $X\setminus B$. Clearly, each $C_i$ is contained in some end $V_i$ and that not two distinct components of $C_i$'s are in the same end.\\
Seeking contradiction, assume that $C_i\notin V$ for all $1\leq i\leq n$, and for each $1\leq i\leq n$, choose $A_i\in V$ such that $C_i\cap A_i=\emptyset$. Notice that $X\setminus \bigcup\limits_{i=1}^{n}C_i$ is bounded, $\bigcap\limits_{i=1}^{n}A_i\in V$ and $\bigcap\limits_{i=1}^{n}A_i\subset X\setminus \bigcup\limits_{i=1}^{n}C_i$, a contradiction.\\

Suppose $X=\bigcup\limits_{n=1}^\infty B_n$, where $B_n\subset B_{n+1}$ for all $n\in\N$ and $B_n$ is a basis of bounded subsets of $X$ (every bounded subset is contained in some $B_n$). From the above, every decreasing sequence $(C_n)_{n\in\N}$ of unbounded components of $X\setminus B_n$ is contained in exactly one end and every end contains exactly one decreasing sequence $(C_n)_{n\in\N}$ of unbounded components of $X\setminus B_n$.
\end{proof}

\begin{Lemma}\label{ExhaustingSequenceOfConnectedCompactWithConnecInterior}
Let $X$ be a Freudenthal space (see \ref{DefFreundenthalspace}). $X$ admits an exhausting sequence $(K_n)_{n\in \mathbb{N}}$ of compact subsets such that both $K_n$ and $int(K_n)$ are connected. 
\end{Lemma}
\begin{proof}
Since $X$ is $\sigma$-compact, it can be written as a countable union of compact subsets, say $X:=\bigcup\limits_{n\in\mathbb{N}} L_n$. As $X$ is locally compact and locally connected, each $L_n$ can be covered by finitely many open connected and relatively compact subsets. Therefore $X$ is a countable union of open connected relatively compact subsets, say $X:=\bigcup\limits_{n\in\mathbb{N}} O_n$. Now, we inductively construct an increasing sequence $(U_n)_{n\in \mathbb{N}\cup\{0\}}$ of open connected relatively compact subsets satisfying $cl_{X}(U_n)\subset U_{n+1}$ for all $n\in \mathbb{N}\cup\{0\}$, and $X:=\bigcup\limits_{n\in\mathbb{N}} U_n$. Put $U_0:=\emptyset$, and assume that $U_1,\ldots,U_n$ are open connected relatively compact subsets such that $cl_{X}(U_i)\subset U_{i+1}$ for all $1\leq i<n$. Since $cl_{X}(U_n)$ is compact, there exist $O_{n_1},\ldots,O_{n_m}$ open connected relatively compact subsets such that $cl_{X}(U_n)\subset \bigcup\limits_{i=1}^{m}O_{n_i}$ and $cl_{X}(U_n)\cap O_{n_i}\neq\emptyset$ for all $1\leq i\leq m$. Put $U_{n+1}:=\bigcup\limits_{i=1}^{m}O_{n_i}$; since $X$ is connected, $cl_{X}(U_n)$ is properly contained in $U_{n+1}$. Moreover, $U_{n+1}$ is open connected relatively compact subset. Finally, since $cl_{X}(\bigcup\limits_{n\in\mathbb{N}} U_n)=\bigcup\limits_{n\in\mathbb{N}} cl_{X}(U_n)=\bigcup\limits_{n\in\mathbb{N}} U_n$, and $X$ is connected, so $X=\bigcup\limits_{n\in\mathbb{N}} U_n$.

Now, put $K_{n}:=cl_{X}(U_n)$ for all $n\in\N$. Notice that for any open subset $U\subset X$, one has $U\subset int(cl_X(U))$. Notice that: $$U_n\subset int(cl_X(U_n))=int(K_n)\subset K_n=cl_X(U_n)\subset U_{n+1}\subset int(cl_X(U_{n+1}))=int(K_{n+1}),$$ and therefore, $int(K_n)$ is connected for all $n\in\N$, and $(K_n)_{n\in \mathbb{N}}$ is an exhausting sequence consists of compact connected subsets. 
\end{proof}

\begin{Theorem}
Let $X$ be a Freudenthal space. There exists a scale $S_X$ and a connectivity structure $C_X$ such that the ends of $(X,S_X,C_X)$ are in one-to-one correspondence with the Freundenthal ends of $X$.
\end{Theorem}
\begin{proof}
Let $S_X$ be the family of all relatively compact subsets of $X$, and $K_0:=\emptyset$. By \ref{ExhaustingSequenceOfConnectedCompactWithConnecInterior}, we can find an exhausting sequence $(K_n)_{n\in \mathbb{N}\cup\{0\}}$ of compact subsets such that both $K_n$ and $int(K_n)$ are connected. Without loss of generality, we may assume that every component of $X\setminus K_1$ is unbounded component. Put $U_n:=int(K_{n+2})\setminus K_n$ and let $C_X$ to be the connectivity structure generated by $\mathcal{U}:=\{U\subset X:U \text{ is a component of } U_n, \text{ for some }n\in\N\}$, i.e. $C\in C_X$ if and only if $C$ is a singleton or every two points $x,y$ of $C$ can be joined by a chain $C_i\in \mathcal{U}$, $1\leq i\leq n$, i.e. $x\in C_1$, $y\in C_n$,
and $C_i\cap C_{i+1}\ne\emptyset$ for each $i < n$. Since $X$ is locally connected, every $U\in\mathcal{U}$ is open connected and relatively compact subset. The ends of $(X,S_X,BA_X)$, where $BA_X$ is the Boolean algebra induced by the connectivity structure $C_X$, are in one-to-one correspondence with the Freundenthal end of $X$.
\end{proof}

\section{Scaled linear spaces}
In this section we introduce analogs of eigenvectors in the context of scaled Boolean algebras.
\begin{Definition}
Given a scaled Boolean algebra $(X,S_X,BA_X)$, a function $f:BA_X \to BA_X $ is a \textbf{$S_X$-linear operator} if:\\
$1.$ $f(C\cup D)\equiv f(C)\cup f(D)$ mod $S_X$
for any $C,D\in BA_X $.\\
$2.$ $f(B)$ is bounded if $B$ is bounded.\\
$3.$ $f(X)\equiv X$ mod $S_X$.
\end{Definition}

\begin{Definition}
A \textbf{scaled linear space} is a quadruple $(X,S_X,BA_X,\mathcal{F})$, where $(X,S_X,BA_X)$ is a scaled Boolean algebra with no internal ends and $\mathcal{F}$ is a non-empty family of $S_X$-linear operators on $BA_X$. When $BA_X:=2^X$ is the powerset Boolean algebra, we write $(X,S_X,\mathcal{F})$ for simplicity. 
\end{Definition}

It is well-known that, for a symmetric 2x2 matrix $M$, if $v$ is an eigenvector of $M$, then so is any $w\ne \vec{0}$ perpendicular to $v$. The following is an analog of that phenomenon.
\begin{Definition}
Given a $S_X$-linear operator $f$, 
$f$ is \textbf{symmetric} if, for any $A\in BA_X$, $f(A)\equiv A$ mod $S_X$ implies $f(X\setminus A)\equiv X\setminus A$ mod $S_X$.

Given a family $\mathcal{F}$ of $S_X$-linear operators, $\mathcal{F}$ is \textbf{symmetric}
if for any $A\in BA_X$, $f(A)\equiv A$ mod $S_X$ for all $f\in \mathcal{F}$ implies
$f(X\setminus A)\equiv X\setminus A$ mod $S_X$ for all $f\in \mathcal{F}$.

Given a family $\mathcal{F}$ of $S_X$-linear operators, an element $A$ of $BA_X$ is an \textbf{eigenset} of $\mathcal{F}$ if $f(A)\equiv A$ mod $S_X$ and $f(X\setminus A)\equiv X\setminus A$ mod $S_X$, for all $f\in\mathcal{F}$. 

\end{Definition}

\begin{Proposition} If $C$ and $D$ are eigensets of $\mathcal{F}$, then so are $C\cup D, D\setminus C$ and $C\cap D$.
\end{Proposition}
\begin{proof}
$f(C\cup D)\equiv f(C)\cup f(D)\equiv C\cup D$ mod $S_X$ for any $f\in \mathcal{F}$.

According to our assumptions both $X\setminus C$ and $X\setminus D$ are eigensets of $\mathcal{F}$. Hence their union and its complement $C\cap D$ are eigensets of $\mathcal{F}$. Finally, $D\setminus C=D\cap (X\setminus C)$.
\end{proof}

\begin{Corollary}
Given a scaled linear space $(X,S_X,BA_X,\mathcal{F})$, the family of eigensets of $\mathcal{F}$ induces a scaled Boolean algebra $(X,S_X,BA_{\mathcal{F}})$.
\end{Corollary}

\begin{Definition}
Given a scaled linear space $(X,S_X,BA_X,\mathcal{F})$, an \textbf{eigenset at infinity} of $\mathcal{F}$ is an end of $(X,S_X,BA_{\mathcal{F}})$.
\end{Definition}

\begin{Definition}
Given a scaled linear space $(X,S_X,BA_X,\mathcal{F})$, and $Y\subset X$. A scaled linear space $(Y,B_Y,BA_Y,\mathcal{F}_Y)$ is a \textbf{scaled linear subspace} of $(X,S_X,\mathcal{F})$ if:\\
$1.$ $(Y,B_Y,BA_Y)$ is a scaled sub-Boolean algebra of $(X,S_X,BA_X)$.\\
$2.$ $(Y,B_Y,BA_{\mathcal{F}_Y})$ is a scaled sub-Boolean algebra of $(X,S_X,BA_{\mathcal{F}_X})$.\\
$3.$ Every element $f\in\mathcal{F}_Y$ admits an extensions $f_X\in\mathcal{F}$ 
such that $f(A)= f_X(A)\cap Y$ for each $A\subset Y$.
\end{Definition}

\begin{Definition}
A scaled linear space $(X,S_X,BA_X,\mathcal{F})$ is \textbf{compact at infinity} if it has no eigensets at infinity. Equivalently, the induced scaled Boolean algebra $(X,S_X,BA_{\mathcal{F}})$ is compact at infinity.
\end{Definition}

\begin{Observation}
A scaled Boolean algebra $(X,S_X,BA_X)$ with no internal ends is compact at infinity if and only if the scaled Boolean algebra $(X,S_X,BA_{\mathcal{F}})$ is compact at infinity for any family $\mathcal{F}$ of $S_X$-linear operators.
\end{Observation}
\begin{proof}
Assume that $(X,S_X,BA_X)$ is compact at infinity and $\mathcal{F}$ is any family  of $S_X$-linear operators. Since $BA_{\mathcal{F}}\subset BA_X$, $(X,S_X,BA_{\mathcal{F}})$ must be compact at infinity. Conversely, take $\mathcal{F}:=\{id_X\}$.
\end{proof}

\begin{Example}\label{GroupActionsExample}
Suppose a group $G$ acts on a set $X$. Define $S_X$ as the family of all finite subsets of $X$, $BA_X$ as all subsets of $X$.
We define $ \mathcal{F}$ as the family $f_g, g\in G$, where
$f_g(A):=A\cdot g$.

$\mathcal{F}$ is is a family of $S_X$-linear operators that are symmetric and its eigensets consist of all almost invariant subsets of $X$.
\end{Example}

\begin{proof}
Recall $A$ is almost invariant if the symmetric difference $A\Delta (A\cdot g)$ is finite for each $g\in G$. 
Since $(A^c\Delta (A^c\cdot g^{-1}))\cdot g=(A\cdot g\setminus A)\cup (A\cdot g^{-1}\setminus A)\cdot g$ for all $A\subset X$ and all $g\in G$,
$A$ is almost invariant if and only if $A^c$ is almost invariant.
Consequently, any almost invariant subset $A$ of $X$ is an eigenset of $f_g$ for all $g\in G$.

Suppose $A$ is an eigenset of $\mathcal{F}$ but $A^c\Delta (A^c\cdot g)$ is not finite for some $g\in G$. Without loss of generality, assume $D:=A^c\cdot g\setminus A^c$ is not finite. Notice $D\subset A$
and $D\cdot g^{-1}\subset A^c$,
so no element of $D$ can belong to $A\cap (A\cdot g)$ resulting in $D\subset A\Delta (A\cdot g)$, a contradiction. Thus, each $f_g$ is symmetric. 
\end{proof}

\begin{Example}
Consider a scaled space $(X,S_X,C_X)$ equipped with a connectivity structure such that for each bounded set $B$ the union of bounded $C_X$-components
of $X\setminus B$ is bounded. Let $BA_X$ be the Boolean algebra induced by $C_X$.
Given $B\in S_X$ and given $A\subset X$, put:$$\mathcal{U}_B^A:=\{C\subset X\setminus B:C\text{ is a component of }X\setminus B\text{ and } A\cap C\text{ is unbounded }\};$$ and define:
\begin{equation*}
f_B(A)=\begin{cases}
         A\quad &\text{if} \, A\in BA_X \\
         
         st(A,\mathcal{U}_B^A) \quad &\text{if} \, A\notin  BA_X.
     \end{cases}
\end{equation*}

$\mathcal{F}:=\{f_B\}_{B\in S_X}$ is a family of $S_X$-linear operators and its eigensets
are exactly elements of $BA_X$.
\end{Example}
\begin{proof}
Suppose $B, B'\in S_X$.\\
\textbf{Claim 1}: If $A$ is equivalent mod $S_X$ to the $C$ union of some unbounded components
of $X\setminus B$ and $B\subset B'$, then $A$ is equivalent mod $S_X$ to a union of some unbounded components
of $X\setminus B'$.\\
\textbf{Proof of Claim 1:} Each unbounded component of $X\setminus B'$
is contained in a unique unbounded component of $X\setminus B$.
Let $C'$ be the union of all unbounded components of $X\setminus B'$
contained in $C$. Notice $A\equiv C'$ mod $S_X$. 

\textbf{Claim 2}: $f_B(A_1\cup A_2)\equiv f_B(A_1)\cup f_B(A_2)$ mod $S_X$.\\
\textbf{Proof of Claim 2:} The only interesting case is that of $f_B(A_1)=A_1$
and $A_2$ not being equivalent mod $S_X$ to a union of unbounded components of any $X\setminus B'$. Pick $B'\in S_X$ containing $B$
such that $A_1$ is equivalent to a union of some unbounded components of
$X\setminus B'$. Notice that $f_B(A_1)\cup f_B(A_2)$
is equivalent to the union of unbounded components of $X\setminus B'$
that intersect $A_1\cup A_2$ non-trivially. The same can be said about
 $f_B(A_1\cup A_2)$.

Notice $f_B(D)=\emptyset$ if $D$ is bounded. Hence, each $f_B$ is $S_X$-linear. By the way the family $\mathcal{F}$ is defined, $BA_X=BA_{\mathcal{F}}$.
\end{proof}

\begin{Example}\label{LSExampleOfSLS}
Suppose $X$ is a large scale space. Define $S_X$ as the family of all bounded subsets of $X$.
Given a uniformly bounded cover $\mathcal{U}$ of $X$ and given $A\subset X$ define $f_{\mathcal{U}}(A)$ as the star of $A$ with respect to $\mathcal{U}$.

$ \mathcal{F}$ is a symmetric family of $S_X$-linear operators, its eigensets consist of coarse clopen subsets of $X$ and its eigensets at infinity consist of coarse ends of $X$.
\end{Example}

\begin{proof}
Recall that a large scale space is a set $X$ equipped with a family of covers
(called \textbf{uniformly bounded covers}) $\mathcal{U}$ that is closed under refinement and under taking stars of uniformly bounded covers.

Coarse open subsets $A$ of $X$ (see \cite{MRD}) are those
for which $st(A,\mathcal{U})\cap st(X\setminus A,\mathcal{U})$ is bounded for each uniformly bounded cover $\mathcal{U}$. That is the same as $f_{\mathcal{U}}(A)\equiv A$ for each $\mathcal{U}$.
\end{proof}

\begin{Example}
Given a locally compact topological group $G$, we define $ \mathcal{F}$ as the family $f_K$, $K\subset G$ is compact, as follows:
$f_K(A)=A\cdot K$.

$B_G$ consists of all pre-compact open subsets of $G$.

$ \mathcal{F}$ is symmetric and its eigensets at infinity consist of ends of $X$.
\end{Example}
\begin{proof}
Similar to Example \ref{GroupActionsExample}. 
\end{proof}

\begin{Definition}
Given a scaled linear space $(X,S_X,\mathcal{F})$ and its ends $Ends(X)$, define $(\bar X,S_X,\overline{\mathcal{F}})$ as follows:\\
$1.$ $\bar X:=X\cup Ends(X)$,\\
$2.$ Given $f\in \mathcal{F}$ define an extension $\bar f:2^{\bar X}\to 2^{\bar X}$ as follows:\\
\begin{equation*}
\bar f(A)=\begin{cases}
         f(A)\cup\{V\in Ends(X):A\cap D \text{ is unbounded, for all } D\in V\}\quad &\text{if} \, A\subset X \\
         \{V\in Ends(X):\forall D\in V,\exists V_D\in A \text{ such that } D\in V_D\}\quad &\text{if} \, A\subset Ends(X) \\
         \bar f(A\cap X)\cup\bar f(A\cap Ends(X))  \quad &\text{if} \, A\subset\bar X,\\
     \end{cases}
\end{equation*}
define $\overline{\mathcal{F}}:=\{\bar f:f\in \mathcal{F}\}$.
\end{Definition}

\begin{Proposition}
$(\bar X,S_X,\overline{\mathcal{F}})$ is a scaled linear space that is compact at infinity.
The eigensets of $\overline{\mathcal{F}}$ are unions of two eigensets of the following forms:\\
1. An eigenset $A$ of $(X,S_X,\mathcal{F})$ union all ends $V$ of $(X,S_X,\mathcal{F})$ containing $A$.\\
2. All ends $V$ of $(X,S_X,\mathcal{F})$ containing $C$ for some eigenset $C$ of $(X,S_X,\mathcal{F})$.
\end{Proposition}
\begin{proof}
Given an eigenset $A$ of $(\bar X,S_X,\overline{\mathcal{F})}$ it is clear that $A\cap X$ is an eigenset of $(X,S_X,\mathcal{F})$.

Given an eigenset at infinity $\{A_i\}_{i\in I}$ of $\overline{\mathcal{F}}$, the family 
$\{A_i\cap X\}_{i\in I}$ has the property of every finite intersection being unbounded.
Hence there is an eigenset at infinity $V$ of 
$(X,S_X,\mathcal{F})$ containing all of them
and $V$ is in the intersection of $\{A_i\}_{i\in I}$.

Given an eigenset at infinity $\{A_i\}_{i\in I}$ of $\overline{\mathcal{F}}$, assume the family 
$\{A_i\cap X\}_{i\in I}$ does not have the property of every finite intersection being unbounded. Therefore assume $\{A_i\setminus X\}_{i\in I}$ has the property of every finite intersection being non-empty.

Suppose $\bigcap\limits_{i\in I}(A_i\setminus X)=\emptyset$. For every $V\in \bar X\setminus X$ choose $i(V)\in I$ such that
$V\notin A_{i(V)}$, i.e. there is $C_V\in V$
with no end in $A_{i(V)}$ containing $C_V$.

There is no end containing all $X\setminus C_V$, so $X\setminus B$ is contained in
a finite union of $C_V$'s for some bounded $B$. Therefore each end $W$ of $(X,S_X,\mathcal{F})$ contains at least one of those $C_V$'s and there is no end in the intersection of corresponding $A_{i(V)}$'s, a contradiction.
\end{proof}

\begin{Definition}
Let $(X,S_X,\mathcal{F})$ be a scaled linear subspace of a compact at infinity scaled linear space 
$(\tilde X,S_X,\mathcal{\tilde F})$. $(\tilde X,S_X,\mathcal{\tilde F})$ is a \textbf{compactification at infinity} of $(X,S_X,\mathcal{F})$ if the following conditions are satisfied:\\
$1.$ $\mathcal{\tilde F}$ consists exactly of extensions $\tilde f$ of elements $f$ of $\mathcal{F}$.\\
$2.$ The scaled Boolean algebra $(\tilde X,S_X,BA_\mathcal{\tilde F})$ is a compactification at infinity of $(X,S_X,BA_\mathcal{F})$ and $BA_\mathcal{F}$ is a sub-Boolean algebra of $BA_\mathcal{\tilde F}$.
\end{Definition}

\begin{Theorem}
For each scaled linear space $(X,S_X,\mathcal{F})$, $(\bar X,S_X,\mathcal{\bar F})$ is its maximal compactification at infinity.
\end{Theorem}

\begin{proof}
Given a compactification at infinity
$(\tilde X,S_X,\mathcal{\tilde F})$ of $(X,S_X,\mathcal{F})$ define an extension $f:\bar X\to \tilde X$ of $id_X$ as follows: for each end $V$ of $(X,S_X,\mathcal{F})$ let $f(V)$ be the intersection of all $\tilde A$, $A\in V$.
\end{proof}

\subsection{Topological scaled linear spaces}

\begin{Definition}\label{DefTopScaledLinearSpace}
A \textbf{topological scaled linear space} $(X,S_X,\mathcal{F})$
is a scaled linear space $(X,S_X,\mathcal{F})$ where $X$ has a topology and $\mathcal{F}$ contains the corresponding closure operator $cl$. 
\end{Definition}

\begin{Observation}
In Definition \ref{DefTopScaledLinearSpace} we may talk about multiple topologies as long as the corresponding closure operator is either included in $F$ or adding it does not change the Boolean algebra $B_{\mathcal{F}}$.
For example, it is so for the discrete topology. 
Also, if there is a symmetric $f\in F$ such that $cl(A)\subset f(A)$ and $cl(A)\in BA_X$ for each $A\in BA_X$, then clearly $f(A)\equiv A$ implies $cl(A)\equiv A$, so adding $cl$ to $\mathcal{F}$ does not change the Boolean algebra $B_{\mathcal{F}}$. The reason it is important is that it implies the topology on the space of ends is independent on those topologies on $X$.
\end{Observation}

\begin{Example}
The structures in Section \ref{EndsOfGeoSpaces} easily can include the closure operator. More generally, any time there is $f\in \mathcal{F}$ such that $f(A)$ is a neighborhood of $A$ containing star of the star of $A$ with respect to an open cover of $X$, adding the closure operator does not affect the family of eigensets. Such is the case of large scale spaces
that have a uniformly bounded cover consisting of open subsets (see \ref{LSExampleOfSLS}). It is so for metric spaces.
\end{Example}

\begin{Definition}\label{DefCompactAtInfForTopSc}
A topological scaled linear space $(X,S_X,\mathcal{F})$
is \textbf{compact at infinity} if for every  cover
of $X\mod S_X$ by unbounded open sets there is a finite subcover of $X\mod S_X$. 
\end{Definition}

\begin{Observation}
Notice in the Definition \ref{DefCompactAtInfForTopSc}, in contrast to the case of scaled Boolean algebras (see \ref{DefCompactAtInfBA}), we do not mention points at infinity as
our definition of scaled linear spaces assume they do not exist.
\end{Observation}

\begin{Proposition}
Given a topological scaled linear space $(X,S_X,\mathcal{F})$, the space $(\bar X,S_X,\mathcal{\bar F})$ is also a topological scaled linear space in the topology induced by $\bar{cl}$ and it is compact at infinity.
\end{Proposition}
\begin{proof}
Define a subset $U$ of $\bar X$ to be open if $U\cap X$ is open and, for each $V\in U\setminus X$, there is $A\in V$, $A\subset U$, such that $W\in U$ for every end $W$ containing $A$. We need to show that $U$ is open if and only if $\bar{cl}(\bar X\setminus U)= \bar X\setminus U$.

By the definition of $\bar{cl}$, $\bar{cl}(\bar X\setminus U)$ is the union of $X\setminus U$, of all ends $V$ such that every $A\in V$ intersects $X\setminus U$ along an unbounded subset, and of all ends $W$ such that for every $A\in W$ there is an end in $\bar X\setminus U$ containing $A$.

If $U$ is open and $V\in U\cap \bar{cl}(\bar X\setminus U)\setminus X$, then there is $A\in V$, $A\subset U$, such that $W\in U$ for every end $W$ containing $A$, a contradiction.

Assume $\bar{cl}(\bar X\setminus U)= \bar X\setminus U$. If $U$ is not open, there is $V\in U\setminus X$, such that there is no $A\in V$, $A\subset U$, such that $W\in U$ for every end $W$ containing $A$.

Therefore each $A\in V$ either intersects $X\setminus U$ along an unbounded subset
(not possible), so there is $A\subset U$
in which case there is $W$ containing $A$ but not in $U$. Hence $W\in \bar X\setminus U$
and $V\in \bar{cl}(\bar X\setminus U)$ contradicting $V\in U$.

Suppose $\{U_i\}_{i\in I}$ is a family of open subsets of $\bar X$ covering $\bar X\setminus X$ such that no finite subfamily covers $\bar X\setminus X$. For each end $V$ choose
$i(V)\in I$ and $A_V\in V$ contained in $U_{i(V)}$.

There is no end containing all $X\setminus A_V$, so $\bar X\setminus X$ is contained in
a finite union of $A_V$'s. Therefore each end $W$ of $(X,S_X,\mathcal{F})$ contains at least one of those $A_V$'s and $\bar X\setminus X$ is contained in the corresponding union of $U_{i(V)}$'s, a contradiction.
\end{proof}

\subsection{Coarse scaled linear spaces}

\begin{Definition}\label{DefCoarseScaledLinearSpace}
A \textbf{coarse scaled linear space} $(X,S_X,\mathcal{F})$
is a scaled linear space $(X,S_X,\mathcal{F})$ where $X$ has a coarse structure and $\mathcal{F}$ contains the corresponding star operators $st$ with respect to uniformly bounded covers of $X$. 
\end{Definition}

\begin{Observation}
In Definition \ref{DefCoarseScaledLinearSpace} we may talk about multiple coarse structures as long as the corresponding star operators are either included in $\mathcal{F}$ or adding them does not change the Boolean algebra $B_{\mathcal{F}}$.
 The reason it is important is that it implies the topology on the space of ends is independent on those coarse structures on $X$.
\end{Observation}

\begin{Example}
The structures in Section \ref{EndsOfGeoSpaces} easily can include the star operators. More generally, any time, given a uniformly bounded cover $\mathcal{U}$ of $X$, there is $f\in \mathcal{F}$ such that $f(A)$ contains star of $A$ with respect to $\mathcal{U}$, adding all the star operators does not affect the family of eigensets. 
\end{Example}

The following can be easily generalized to locally compact spaces.
\begin{Proposition}
Suppose $X$ is a proper metric space.
If $(X,S_X,\mathcal{F})$ is a topological scaled linear space and a coarse scaled linear space, then $X\cup Ends(X)$ is a compactification of $X$ that is dominated by the Higson compactification of $X$.
\end{Proposition}
\begin{proof}
It is sufficient to show that, for any continuous function $f:X\cup Ends(X)\to [0,1]$, its restriction $f|X$ to $X$ is a Higson function (i.e. both continuous and slowly oscillating). 

Suppose $f$ is not slowly oscillating. That means there are two sequences 
$\{x_n\}_{n\ge 1}$ and $\{y_n\}_{n\ge 1}$  in $X$ both diverging to infinity and $|f(x_n)-f(y_n)|\to \epsilon > 0$ but there is $M > 0$ such that $d(x_n,y_n) < M$ for all $n\ge 1$. Reduce it to the case of $\lim f(x_n)=a$,
$\lim f(y_n)=a+4b$, $b > \epsilon/4$ and $f(x_n)\in (a-b,a+b)$, $f(y_n)\in (a+2b,a+5b)$ for all $n\ge 1$. We may find an end $V_1\in cl(\{x_n\}_{n\ge 1})$,
such that $f(V_1)\subset (a-b,a+b)$. Hence there is an eigenset $A_1\in V_1$
of $\mathcal{F}$ such that $f(A_1)\subset [a-b,a+b]$. That eigenset
contains infinitely many elements of $\{x_n\}_{n\ge 1}$.
Look at the corresponding subsequence of $\{y_n\}_{n\ge 1}$
and find an end $V_2$ in its closure such that $f(V_2)\in [a+2b,a+5b]$.
there is an eigenset $A_2\in V_2$
of $\mathcal{F}$ such that $f(A_2)\subset [a+2b,a+5b]$. That eigenset
contains infinitely many elements of $\{y_n\}_{n\ge 1}$.
Since $B(A_2,2M)\setminus A_2$ is bounded, infinitely many elements of
$\{y_n\}_{n\ge 1}$ and infinitely many elements of $\{x_n\}_{n\ge 1}$
are contained in $A_2$, a contradiction.
\end{proof}

\section{Actions of scaled groups}

The following is a natural generalization of locally compact topological groups.
\begin{Definition}\label{DefScaledGroup}
A \textbf{scaled group} is a pair $(G,S_G)$ consisting of a group $G$ and a scale $S_G$ on $G$ \textbf{covering} it such that $B_1\cdot B_2^{-1}\in S_G$ whenever $B_1, B_2\in S_G$.
\end{Definition}

Notice that the family of compact (finite, countable) subsets in any topological (discrete) group forms a scale as in Definition \ref{DefScaledGroup}.

\begin{Proposition}
If $(G,S_G)$ is a scaled group, then it induces four scaled linear spaces:\\
1. $(G,S_G,\cdot S_G)$,
where $\cdot S_G$ is the family of functions $rm_B:2^G\to 2^G$, $B\in S_G$, defined by $rm_B(A)=A\cdot B$.\\
2. $(G,S_G,S_G\cdot)$,
where $S_G\cdot $ is the family of functions $lm_B:2^G\to 2^G$, $B\in S_G$, defined by $lm_B(A)=B\cdot A$.\\
3. $(G,S_G,\cdot G)$,
where $\cdot G$ is the family of functions $rm_g:2^G\to 2^G$, $g\in G$, defined by $rm_g(A)=A\cdot g$.\\
4. $(G,S_G,G\cdot)$,
where $G\cdot $ is the family of functions $lm_g:2^G\to 2^G$, $g\in G$, defined by $lm_B(A)=g\cdot A$.\\
Scaled linear spaces $(G,S_G,\cdot S_G)$ and $(G,S_G,S_G\cdot)$ ($(G,S_G,\cdot G)$ and $(G,S_G,G\cdot)$) have homeomorphic end spaces.
\end{Proposition}
\begin{proof}
Notice that for all $A,B,C\in 2^G$, then $(A\cup C)\cdot B=A\cdot B\cup C\cdot B$ and $B\cdot (A\cup C)=B\cdot A\cup B\cdot C$.

$A$ is an eigenset of $ S_G\cdot$ if and only if, for each $B\in S_G$
there is $K\in S_G$ such that $A\cdot B\cup K=A\cup K$ which is equivalent to $B^{-1}\cdot A^{-1}\cup K^{-1}=A^{-1}\cup K^{-1}$.
That means $A\to A^{-1}$ establishes homeomorphism of end spaces.
\end{proof}

\begin{Definition}\label{DefActionOfSGroups}
An \textbf{action of a scaled group} $(G,S_G)$ on a scaled space $(X,S_X)$
is a group action $(g,x)\to g\cdot x$ such that $B\cdot K$ is bounded in $X$ whenever $B\in S_G$ and $K\in S_X$. It is 
\textbf{cobounded} if there is $K_0\in S_X$ such that $X=G\cdot K_0$.
The action is \textbf{proper} if for every $K\in S_X$ the set $\{g\in G \mid (g\cdot K)\cap K\ne\emptyset\}$ (called the \textbf{stabilizer} of $K$) is bounded in $G$.
\end{Definition}

\begin{Observation}
Notice that any action of a scaled group $(G,S_G)$ on a scaled space $(X,S_X)$ induces two scaled linear spaces:\\
1. One consisting of $m_g:2^X\to 2^X$ ($g\in G$), $m_g(A):=g\cdot A$ is denoted by $(X,S_X,G\cdot)$.\\
2. The other consisting of $m_B:2^X\to 2^X$ ($B\in S_G$, $B\ne\emptyset$), $m_B(A):=B\cdot A$ is denoted by $(X,S_X,S_G\cdot)$.\\
If $S_G$ consists of all finite subsets of $G$, then the two linear scaled spaces have the same eigensets, hence the same end spaces.
\end{Observation}
\begin{proof}
$A$ being an eigenset of $(X,S_X,G\cdot)$ means that for each $g\in G$ the set $(g\cdot A)\Delta A$ is bounded in $X$.
Therefore, if $B\subset G$ is finite and non-empty,
then $(B\cdot A)\Delta A$ is a finite union of bounded subsets of $X$, hence it is bounded.
\end{proof}

\begin{Proposition}
Let $H$ be a subgroup of a scaled group $(G,S_G)$
and $S_H:=S_G\cap H$. \\
1. $(H,S_H)$ is a scaled group.\\
2. The natural actions $(h,x)\to h\cdot x$ and $(h,x)\to x\cdot h^{-1}$ of $(H,S_H)$ on $(G,S_G)$ are proper. Moreover, if $H$ is of bounded index, then the actions are cobounded.
\end{Proposition}

\begin{proof}
1. Is obvious.\\
2.
Recall that $H$ is of bounded index if $G=H\cdot B$, for some bounded subset $B$ of $G$. Therefore, the action is cobounded. Notice the stabilizer of $B$ is contained in $H\cap (B\cdot B^{-1})$, so it is bounded in $H$ and the action is proper.
\end{proof}

We follow a description of coarse spaces (quite often our terminology is that of \textbf{large scale spaces}) as in \cite{DH}. It is equivalent to Roe's definition of those spaces in \cite{Roe lectures}.

Recall that a \textbf{star} $st(x,\mathcal{U})$ of $x\in X$ with respect to a family $\mathcal{U}$ of subsets of $X$ is defined as the union of $U\in \mathcal{U}$ containing $x$. If $A\subset X$, then $st(A,\mathcal{U}):=\bigcup\limits_{x\in A}st(x,\mathcal{U})$. Given two families $\mathcal{U},\mathcal{V}$ of subsets of $X$,
$st(\mathcal{U},\mathcal{V})$ is defined as the family $st(A,\mathcal{V})$, $A\in \mathcal{U}$.

\begin{Definition}
A \textbf{large scale space} is a set $X$ equipped with a family $\mathbb{LSS}$ of covers (called \textbf{uniformly bounded} covers) satisfying the following two conditions:\\
1. 
$st(\mathcal{U},\mathcal{V})\in \mathbb{LSS}$ if $\mathcal{U},\mathcal{V}\in \mathbb{LSS}$.\\
2. If $\mathcal{U}\in \mathbb{LSS}$ and every element of $\mathcal{V}$ is contained in some element of $\mathcal{U}$, then $\mathcal{V}\in \mathbb{LSS}$.

Sets which are contained in an element of $\mathcal{U}\in \mathbb{LSS}$ are called \textbf{bounded}.
Thus, each large scale structure on $X$ induces a natural scale on $X$.
\end{Definition}

The following is in some sense dual to the famous \v Svarc-Milnor Lemma (see \cite{BDM1} and \cite{MRD} for other generalizations of the classical version of that lemma). Proposition \ref{ScaledActionsInduceLargeScaleStructure} establishes existence of certain large scale structures and our version of the \v Svarc-Milnor Lemms (see
\ref{Svarc-MilnorForScaledGroups}) establishes uniqueness of such a structure.
\begin{Proposition}\label{ScaledActionsInduceLargeScaleStructure}
If a scaled group $(G,S_G)$ acts properly and coboundedly on a scaled space $(X,S_X)$, then it induces a large scale structure on $X$ that is coarsely equivalent to that on $G$.
\end{Proposition}
\begin{proof}
Consider refinements of covers of $X$ of the form $C_B:=\{g\cdot B\}_{g\in G}$
for some $B\in S_X$ (hence $G\cdot B=X$). To see that $st(C_B,C_B)$ is such a refinement,
look at all $g\in G$ such that for some fixed $h\in G$, $(g\cdot B)\cap (h\cdot B)\ne\emptyset$. What that means is that $h^{-1}\cdot g\in B'$,
where $B'$ is the stabilizer of $B$. Hence $g\in h\cdot B'$
and the the cover $st(C_B,C_B)$ refines $C_{B'\cdot B}$. 
Indeed, given $h\in G$, the union of all $g\cdot B$ intersecting
$h\cdot B$ is contained in $h\cdot B'\cdot B$.

As $G$ acts on itself, the above recipe says that the large scale structure on $G$ consists of refinements of covers of $G$ of the form $C_B:=\{g\cdot B\}_{g\in G}$
for some $B\in S_G$. Pick $x_0\in B_0$, where $G\cdot B_0=X$, and consider $f:G\to X$, $f(g):=g\cdot x_0$. $f$ sends uniformly bounded covers of $G$ to uniformly bounded covers of $G\cdot x_0$ and $f^{-1}$ does the same. Indeed, $f(g\cdot B)$,
$B\in S_G$, is equal to $g\cdot B\cdot x_0$.
Conversely, if $f(g)\in h\cdot B$, $B\in S_X$ containing $x_0$,
then $g\cdot x_0\in h\cdot B$ and $h^{-1}\cdot g$ belongs to the stabilizer $B'$ of $B$. Thus, $f^{-1}(C_B)$ refines $C_{B'}$.

Moreover,
$st(G\cdot x_0,C_{B_0})=X$, so the inclusion $G\cdot x_0\to X$ is a coarse equivalence.
\end{proof}

\begin{Remark}\label{ScaledVsLargeScaleGroups}
In view of \ref{ScaledActionsInduceLargeScaleStructure} every scaled group $(G,S_G)$ has a natural large scale structure. In \cite{MRD} such groups are called \textbf{large scale groups}.
\end{Remark}

\begin{Theorem} \label{Svarc-MilnorForScaledGroups} 
Given a scaled group $(G,S_G)$ consider the large scale structure $\mathcal{LSS}$ on $G$ induced by refinements of the covers $C_B:=\{g\cdot B\}_{g\in G}$, $B\in S_G$. 
Suppose $(G,S_G)$ acts by uniform coarse equivalences on a large scale space $X$. If the action is proper and cobounded, then for each $x_0\in X$ the map $g\to g\cdot x_0$ is a coarse equivalence from $(G,\mathcal{LSS})$ to $X$.
\end{Theorem}
\begin{proof}
Choose $B_0\in S_X$, $S_X$ being the family of all bounded subsets of $X$, such that $G\cdot B_0=X$
$(G,S_G)$ acts by uniform coarse equivalences on a large scale space $X$ means that, for each uniformly bounded cover $\mathcal{U}$ of $X$,
there is a uniformly bounded cover $\mathcal{V}$ of $X$ so that each $g\cdot \mathcal{U}$, $g\in G$, refines $\mathcal{V}$.
Each $U\in \mathcal{U}$ contains $x_U\in U$ belonging to $g_U\cdot B_0$ for some $g_U\in G$. Therefore
$g_U^{-1}\cdot U$ is contained in an element $V(U)$ of $\mathcal{V}$
and intersecting $B_0$. Hence, $ g_U^{-1}\cdot U\subset W:=st(B_0,V)$ and $U\subset g_U\cdot W\in S_X$.
That means the large scale structure on $X$ is identical to the one described in \ref{ScaledActionsInduceLargeScaleStructure}.
Apply \ref{ScaledActionsInduceLargeScaleStructure}.
\end{proof}

The following result mixes right and left multiplication actions of a group $G$ on itself. It says that the two ways of defining end spaces on a scaled groups are identical.
\begin{Proposition}\label{CClopenEqualEigenset}
Given a scaled group $(G,S_G)$ consider the large scale structure $\mathcal{LSS}$ on $G$ induced by refinements of the covers $C_B:=\{B\cdot g\}_{g\in G}$, $B\in S_G$. $A$ is coarsely clopen in $\mathcal{LSS}$ if and only if $A$ is an eigenset of $S_G\cdot$.
Therefore, the end space of $(G,S_G,S_G\cdot)$ is identical
with the end space of the coarse space $(G,\mathcal{LSS})$.
\end{Proposition}
\begin{proof}
Let us show $st(A,C_B)\setminus A=B\cdot B^{-1}\cdot A\setminus A$. That suffices to prove \ref{CClopenEqualEigenset}.
$g\in st(A,C_B)\setminus A$ if and only $g\notin A$ and there is $h\in G$ such that $g\in B\cdot h$ and $(B\cdot h)\cap A\ne\emptyset$. Thus, $h\in B^{-1}\cdot A$ and $g\in B\cdot B^{-1}\cdot A\setminus A$.
\end{proof}

In view of Propositions \ref{CClopenEqualEigenset} and \ref{ScaledActionsInduceLargeScaleStructure} the following question seems relevant:

\begin{Question}
Suppose a scaled group $(G,S_G)$ acts properly, transitively, and coboundedly on a scaled space $(X,S_X)$. Are end spaces of $(G,S_G,S_G\cdot)$ and $(X,S_X,S_G\cdot)$ homeomorphic.

What if we replace transitive actions by cobounded actions?
\end{Question}

In the remainder of this section we will rephrase results from \cite{MRD}
related to large scale groups in terms of scaled groups using Remark \ref{ScaledVsLargeScaleGroups} and Proposition \ref{CClopenEqualEigenset}.

\begin{Definition}\label{LocallyBoundedGroupDef}
A scaled group $(G,S_G)$ is \textbf{locally bounded} if for every bounded subset $B$ of $G$ the subgroup $\lbrack B\rbrack$ of $G$ generated by $B$ is bounded.
\end{Definition}

\begin{Proposition}\label{LocallyBoundedGroupsCase}\cite{MRD} 
Suppose the scale $S_G$ of the scale group $G$ has a countable basis.
If $G$ is an unbounded and locally bounded group, then its number of ends of $(G,S_G)$ is infinite.
\end{Proposition}

\begin{Definition}
A scaled group $(G,S_G)$ is of \textbf{bounded geometry} if there is a  bounded set $K$ such that for every bounded subset $B$ of $G$ there are elements $g_i\in G$, $i\leq k$, so that $B\subset \bigcup\limits_{i=1}^k g_i\cdot K$.
\end{Definition}

\begin{Theorem}\label{SpecialStallingsTwoEnds}\cite{MRD}
Let $(G,S_G)$ be a boundedly generated scaled group of bounded geometry. If $(G,S_G,S_G\cdot)$ has two ends, then $G$ contains an unbounded cyclic subgroup of bounded index.
\end{Theorem}

\begin{Theorem}
Let $(G,S_G)$ be a $\sigma$-bounded scaled group of bounded geometry. If $(G,S_G,S_G\cdot)$ has two ends, then $G$ is boundedly generated.
\end{Theorem}

\section{Ends of geodesic spaces}\label{EndsOfGeoSpaces}
Given a geodesic space $X$ and $x_0\in X$ the most natural operation on a subset $A$ of $X$ is its \textbf{cone} $Cone(A,x_0)$ defined as the union of all geodesics $[x_0,x]$ joining $x_0$ and $x\in A$. There are two ways to mix it up: one is first to take a ball around $A$, then the cone. The other is to take the ball of the cone. We will show that the two processes yield the same end spaces.

\begin{Definition}\label{DefGeodesicOperators}
One family of operators, $\mathcal{F}_{cb}$ consists of $cb_r:2^X\to 2^X$, $r > 0$, defined as
$cb_r(A):=Cone(B(A,r),x_0)$.\\
The other family of operators $\mathcal{F}_{bc}$ consists of $bc_r:2^X\to 2^X$, $r > 0$, defined as
$bc_r(A):=B(Cone(A,x_0),r)$.

To show that they have the same families of eigensets with respect to
$S_X$ defined as all open bounded subsets of $X$, we introduce a combination of the above families: $g_r(A):=B(cb_r(A),r)$.
\end{Definition}

\begin{Observation}
As $A\subset  cb_r(A)\cap bc_r(A)\subset cb_r(A)\cup bc_r(A)\subset
g_r(A)$ for any $A\subset X$, any eigenset of the family $\mathcal{F}=\{g_r\}_{r > 0}$ is an eigenset
of the other two families. Our goal is to show all three families have the same eigensets and they do not depend on the basepoint $x_0$.
\end{Observation}

\begin{Definition}
Given a pointed geodesic space $(X,x_0)$, its \textbf{geodesic eigenset}
is a subset $A$ such that $A$ (and therefore $X\setminus A$ as well) is an eigenset of $(X,S_X,\mathcal{F})$, where $S_X$ consists of all open bounded subsets of $X$ and $\mathcal{F}=\{g_r\}_{r > 0}$ is as in \ref{DefGeodesicOperators}. 
\end{Definition}

\begin{Definition}
Given a pointed geodesic space $(X,x_0)$, its \textbf{geodesic ends}
are defined as the ends of
$(X,S_X,\mathcal{F})$, where $S_X$ consists of all open bounded subsets of $X$ and $\mathcal{F}=\{g_r\}_{r > 0}$. 
\end{Definition}

\begin{Proposition}
If $X$ is a geodesic space, then for any bounded subset $K$ of $X$
any union $C$ of components of $X\setminus K$ is a geodesic eigenset of $X$.
\end{Proposition}
\begin{proof}
Without loss of generality, we may assume that $C$ is unbounded. Given $r > 0$ suppose $x_1\in g_r(C)\setminus B(Cone(B(K,2r),x_0),2r)$. Choose $x_2\in Cone(B(C,r),x_0)$ at the distance from $x_1$ less than $r$. 
There is a geodesic $[x_0,y]$ containing $x_2$, where $y\in B(C,r)$.
Choose $z\in C$, $d(y,z) < r$. Notice that implies $y\in C$ as any geodesic
$[y,z]$ is outside of $K$. Indeed, if $t\in [y,z]\cap K$,
then $y\in B(K,r)$ and $x_2\in Cone(B(K,r),x_0)$
resulting in $x_1\in B(Cone(B(K,r),x_0),r)$, a contradiction.

Now, the sub-geodesic $[x_2,y]$ of $[x_0,y]$
is outside of $B(K,2r)$ as well (otherwise $x_1\in B(Cone(K,x_0),2r)$), so it is contained in $C$. Hence, any geodesic $[x_1,x_2]$ misses $K$,
so we can conclude $x_1\in C$.

Finally, if $B:=B(Cone(B(K,2r),x_0),2r)$, then
$C\setminus B\subset g_r(C)\setminus B\subset C$ which implies
$C\equiv g_r(C)$ mod $S_X$.

Since $(X\setminus K)\setminus C$ is also a union of some components of $X\setminus K$, $C$ is an eigenset of $(X,S_X,\mathcal{F})$.
\end{proof}

\begin{Proposition}\label{GeoEigensetIsEquiToUnionOComponents}
If $X$ is a geodesic space and $A\equiv cb_1(A)$ or $A\equiv bc_1(A)$ mod $S_X$,
then there is a bounded closed subset $K$ of $X$ such that $A$ is equivalent mod $S_X$ to
some union of components of $X\setminus K$.
\end{Proposition}
\begin{proof}
Without loss of generality, we may assume that $A$ is unbounded. Pick $m > 0$ such that $cb_1(A)\setminus B(x_0,m)=A\setminus B(x_0,m)$
or  $bc_1(A)\setminus B(x_0,m)=A\setminus B(x_0,m)$.
Suppose $x\in A$ belongs to a component $C$ of $X\setminus cl(B(x_0,m+1))$ and assume $y\in C$. Choose a piece-wise geodesic path $P$ from $x$ to $y$ in $C$.
If $z\in P$ and $d(x,z) < 1$, then $z\in cb_1(A)$ (or $z\in bc_1(A)$), so $z\in A$
as $z\notin cl(B(x_0,m+1))$. By induction on finitely many elements of $P$, $y\in A$.
That means $A$ contains all the components of $X\setminus cl(B(x_0,m+1))$ that intersect $A$. In particular, $A$ is equivalent to the union of such components.
\end{proof}

\begin{Corollary}
If $X$ is a geodesic space, then all three families from \ref{DefGeodesicOperators} have the same eigensets, namely unions of some components of $X\setminus K$, $K$ being bounded. In particular, those eigensets do not depend on the basepoint in $X$.
\end{Corollary}

\begin{Corollary}
If $X$ is a proper geodesic space, then its geodesic ends correspond to Freundenthal ends of $X$.
\end{Corollary}

Given a metric space $(X,d)$, the \textbf{Gromov product} of $x$ and $y$ with respect to $a\in X$ is defined by
$$
\left< x,y\right>_a=\frac{1}{2}\big(d(x,a)+d(y,a)-d(x,y)\big).
$$

Recall that metric space $(X,d)$ is (Gromov) $ \delta-$\textbf{hyperbolic} if it satisfies the $\delta/4$-inequality:
$$
\left< x,y\right>_{a} \geq \min \{\left< x,z\right>_{a},\left< z,y \right> _{a}\}-\delta/4, \quad \forall x,y,z,a\in X.$$

$(X,d)$ is \textbf{Gromov hyperbolic} if it is $ \delta-$hyperbolic for some $\delta > 0$.

The assumptions in the next Proposition are chosen specifically so that it applies to both visual hyperbolic spaces and visual CAT(0)-spaces.

\begin{Proposition}
Suppose $X$ is a visual geodesic space (i.e. there is $x_0\in X$ such that the union of all geodesic rays emanating from $x_0$ equals $X$), $\sim$ is an equivalence relation on
the set of geodesic rays emanating from the basepoint $x_0\in X$, and $\mathcal{T}$ is a topology an the set $\partial X$ of equivalence classes induced by $\sim$. 
Given $A\subset \partial X$ define $Cone(A)$ as the union of all geodesic rays $r$ so that $[r]\in A$. Suppose the following conditions are satisfied:\\
1. If a sequence of rays $r_n$ converges pointwise to $r_0$,
then $[r_n]$ converges to $[r_0]$ in $\mathcal{T}$.\\
2. Given a sequence $[r_n]\in \partial X$, $n\ge 1$,
there is a subsequence $n(k)$ of positive integers
and rays $s_k$ converging pointwise to a ray $r$ so that $[s_k]=[r_{n(k)}]$
for each $k\ge 1$.\\
3. If $[r]$ belongs to the closure $cl(A)$ of $A\subset \partial X$, then there are exist rays $r_n$, $n\ge 0$,
such that $[r_0]=[r]$, $[r_n]\in A$ for $n\ge 1$, and $r_n$ converges pointwise to $r_0$.\\
4. Given two geodesic rays $r$ and $s$ at $x_0$, $r\sim s$
if and only if there is a sequence $x_n\in r$ diverging to infinity
and $M > 0$ such that $dist(x_n,s) < M$ for each $n\ge 1$.

If $A$ is a clopen subset of $\partial X$, then both $Cone(A)$ and $X\setminus Cone(A)$ are geodesic eigensets of $X$. 

Conversely, given $D\subset X$ such that $D$ and $X\setminus D$ are  eigensets of $(X,S_X,\mathcal{F})$, then there is a clopen subset $A$ of  $\partial X$ of $X$ such that $D\equiv Cone(A)$ mod $S_X$.
\end{Proposition}
\begin{proof}
Let $C:=Cone(A)$.
Suppose $g_r(C)\setminus C$ is not bounded for some $r > 0$.
We can choose a sequence $x_n\in g_r(C)\setminus C$ diverging to infinity
such that some geodesic rays $l_n$ containing $x_n$ converge pointwise to a ray $l$. Therefore $[l]\notin A$.
Choose points $y_n\in C$ satisfying $d(x_n,y_n)< 2r$ for large $n$.
We may choose, by passing to subsequences, rays $l'_n\in A$ containing $y_n$ and converging pointwise to a ray $l'$ equivalent to $l$, hence not in $C$,
a contradiction. 

Let us show that an unbounded geodesic eigenset is, up to a bounded set, 
equivalent to $Cone(A)$ for some clopen $A$. Suppose $D\subset X$ is a geodesic eigenset of $(X,S_X,\mathcal{F})$. By \ref{GeoEigensetIsEquiToUnionOComponents}, we can choose $M > 0$ such that $D\setminus B(x_0,M)$ and $(X\setminus D)\setminus B(x_0,M)$ are union of components of $X\setminus B(x_0,M)$. Given $r > 0$ and given a geodesic ray $l$ at $x_0$, $l\setminus B(x_0,M)$
intersects either $D$ or $X\setminus D$ but not both. Put $A:=\{[l]:l\text{ is a geodesic ray at } x_0\text{ and } l\setminus B(x_0,M)\cap D\ne\emptyset\}$. Notice that if two geodesic rays $t$ and $s$ at $x_0$, satisfy $t\sim s$ and $t\cap D$ contains a geodesic ray
emanating from a point in $D$, then $s\cap D$ contains a geodesic ray
emanating from a point in $D$. 
Indeed, by Condition 4), there is a sequence $x_n\in t$ diverging to infinity
and $M > 0$ such that $dist(x_n,s) < M$ for each $n\ge 1$.
Consider $r:=M+1$ and conclude that $s\cap D$ contains a geodesic ray
emanating from a point in $D$. Therefore, $D\equiv Cone(A)$, and similarly $X\setminus D\equiv Cone(\partial X\setminus A)$. It remains to show that $A$ is closed in $\partial X$ as that implies it is clopen in $\partial X$.
It is so by Condition 3).
\end{proof}

\begin{Corollary} Suppose $X$ is a proper geodesic space.
If $X$ is Gromov hyperbolic and visual, then its geodesic ends correspond to the components of the Gromov boundary $\partial_h(X)$ of $X$, 
\end{Corollary}

\begin{Corollary} Suppose $X$ is a proper geodesic space.
If $X$ is a visual CAT(0) space, then the geodesic ends of $X$ correspond to components of the visual boundary of $X$.
\end{Corollary}

\begin{Question}
Given a metric space $X$ and $p\in X$ one can introduce $S_X$-operators ($S_X$ being the family of open bounded subsets of $X$) $f_r$, $r > 0$,
as 
$$f_r(A)=\{x\in X \mid \left< x,y\right>_p > r \textrm{ for some y in A}\}.$$
What can be said about the resulting ends of $X$?
\end{Question}


\begin{thebibliography}{99}
\bibitem{BH} M. Bridson and A. Haefliger, \emph{Metric spaces of non-positive curvature}, Springer- Verlag, Berlin, 1999.


\bibitem{BDM1}   N.Brodskiy, J.Dydak, and A.Mitra, {\em    Svarc-Milnor Lemma: a proof by definition},  Topology Proceedings 31 (2007)

\bibitem{BDM} N.Brodskiy, J.Dydak, and A.Mitra, \emph{Coarse structures and group actions},
Colloquium Mathematicum 111 (2008), 149--158.

\bibitem{Corn} Yves Cornulier, \emph{On the space of ends of infinitely generated groups}, Topology Appl. 263 (2019) 279-298

\bibitem{DD} W. Dicks, M. J. Dunwoody, \emph{Groups acting on graphs}, Vol. 17, Cambridge University Press, 1989.

\bibitem{DK} C. Drutu, M. Kapovich, \emph{Geometric group theory}, Colloquium publications, Vol. 63, American Mathematics Society (2018).

\bibitem{DH}   J.Dydak and C.Hoffland, \emph{An alternative definition of coarse structures},
Topology and its Applications 155 (2008) 1013--1021

\bibitem{Engel}
R. Engelking, \emph{Theory of dimensions finite and infinite}, Sigma Series in Pure Mathematics,
vol. 10, Heldermann Verlag, 1995.

\bibitem{Grom}
M. Gromov, \emph{Asymptotic invariants for infinite groups}, in
Geometric Group Theory, vol. 2, 1--295, G. Niblo and M. Roller,
eds., Cambridge University Press, 1993.

\bibitem{LV} Arielle Leitner and Federico Vigolo, \emph{An Invitation to Coarse Groups}, preprint (July 9, 2021)

\bibitem{MRD} Yuankui Ma, Hussain Rashed and Jerzy Dydak, \emph{Ends of large scale groups},
in preparation

 \bibitem{Peschke} Georg Peschke, \emph{The Theory of Ends},
Nieuw Archief voor Wiskunde, 8 (1990), 1--12

\bibitem{Sp} E. Specker, \emph{Endenverb\"ande von Raumen und Gruppen}, Math. Ann. 122, (1950). 167--174.

\bibitem{St}   J. Stallings, \emph{On torsion-free groups with infinitely many ends}, Annals of Mathematics (1968), 312--334.

\bibitem{Stone} M.Stone, \emph{The Theory of Representations of Boolean Algebras} Transactions of the American Mathematical Society 40 (Jul. 1936), pp. 37-111.

\bibitem{HR} Hussain Rashed, \emph{The Stone's Representation Theorems and compactifications of a discrete space},	arXiv:2112.07821 [math.GN].

\bibitem{Roe lectures}
J. Roe, \emph{Lectures on coarse geometry}, University Lecture
Series, 31. American Mathematical Society, Providence, RI, 2003.

\bibitem{Zippin}
Leo Zippin, \emph{Two-Ended Topological Groups}, Proceedings of the American Mathematical Society, Jun., 1950, Vol. 1, No. 3 (Jun., 1950), pp. 309-315






\end{thebibliography}
\end{document}